\documentclass{amsart}

\usepackage{amssymb}
\usepackage{amsmath}
\usepackage{enumitem}
\usepackage[bookmarks=false]{hyperref}

\usepackage{tikz-cd}
\usepackage{pgfplots}

\title{Knobbly but nice}
\author{Neil Dobbs}
\address{School of Mathematics and Statistics, UCD, Belfield, Dublin 14, Ireland.}
\email{neil.dobbs@ucd.ie}

\newcommand\re{{\Re}}
\newcommand\im{{\Im}}

\newcommand\eps{\varepsilon}
\DeclareMathOperator{\dist}{dist}

\newcommand\ccc{\mathbb{C}}
\newcommand\R{\mathbb{R}}
\newcommand\Z{\mathbb{Z}}
\newcommand\C{\mathbb{C}}
\newcommand\N{\mathbb{N}}
\newcommand\bbS{\mathbb{S}}

\newtheorem{mainthmnum}{Theorem}

\newtheorem{thm}{Theorem}[section]
\newtheorem{dfn}[thm]{Definition}
\newtheorem{conj}[thm]{Conjecture}
\newtheorem{lem}[thm]{Lemma}
\newtheorem{prop}[thm]{Proposition}
\newtheorem{cor}[thm]{Corollary}
\newtheorem*{cornonum}{Corollary}

\newtheorem{rem}[thm]{Remark}
\newtheorem{example}[thm]{Example}

\numberwithin{equation}{section}

\usetikzlibrary{decorations.markings}
\tikzset{negated/.style={
        decoration={markings,
            mark= at position 0.5 with {
                \node[transform shape] (tempnode) {$\backslash$};
            }
        },
        postaction={decorate}
    }
}

\newcommand\cB{\mathcal{B}}
\newcommand\cK{\mathcal{K}}

\newcommand\cC{{\mathcal C}}

\newcommand\rmthm[1]{{Theorem~\ref{mthm#1}}}

\newcommand\rprop[1]{{Proposition~\ref{prop#1}}}
\newcommand\rlem[1]{{Lemma~\ref{lem#1}}}
\newcommand\rcor[1]{{Corollary~\ref{cor#1}}}

\begin{document}
\begin{abstract}
Our main result states that, under an exponential map whose Julia set is the whole complex plane, on each piecewise smooth Jordan curve there is a point whose orbit is dense. 
    This has consequences for the boundaries of \emph{nice} sets, used in induction methods to study ergodic and geometric properties of the dynamics. 
\end{abstract}
    
\maketitle

\section{Introduction}

%
We consider dynamics of exponential maps, work in the complex plane and show the following.
\begin{mainthmnum} \label{mthmPoint}
    Let $f : z \mapsto \lambda e^z$. If  the Julia set of $f$ is the whole complex plane, on each  piecewise-smooth Jordan curve there is a point whose orbit is dense. 
\end{mainthmnum}
If the Julia set is not the whole plane, no point has a dense orbit. 
\begin{mainthmnum} \label{mthmDicho}
    Let $f : z \mapsto \lambda e^z$. If a Jordan curve surrounds a point of the Julia set, one of the following hold:
    \begin{enumerate}
        \item
            The orbit of the Jordan curve is dense.
        \item
            There is an uncountable set of points on the Jordan curve where the curve is non-smooth.
    \end{enumerate}
\end{mainthmnum}
\begin{dfn}
A \emph{nice set} is an open set $U$ such that $f^n(\partial U) \cap U = \emptyset$ for all $n\geq0$. 
\end{dfn}
\begin{cornonum}
    No bounded  nice set with piecewise-smooth boundary can intersect the Julia set of $f : z \mapsto \lambda e^z$. 
\end{cornonum}
A \emph{piecewise-smooth Jordan curve} will be the image $h(\mathbb{S}^1)$ of the unit circle by a homeomorphic map $h : \mathbb{S}^1 \to h(\mathbb{S}^1) \subset \C$ which is continuously differentiable, with non-zero derivative, except possibly at  a finite number of points; $\lambda$ is a complex number.   
A Jordan curve $H = h(\bbS^1)$ admits different parametrisations. Given $z \in H$, if there exist a sub-arc $W \subset H$ containing $z$ and a $C^1$ diffeomorphism $g : (0,1) \to W$, then we say $H$ is \emph{smooth at $z$}; otherwise $H$ is \emph{not smooth at} $z$. 

The stimulus for this note came from several sources. 
The following conjecture was conveyed to the author by Lasse Rempe-Gillen. 
\begin{conj}
    Every line segment intersects the escaping set of $z \mapsto e^z$. 
\end{conj}
It was perhaps inspired by the following long-standing open problem.
\begin{conj}
    The hairs of $f : z \mapsto e^z$ are real-analytic curves. 
\end{conj}
A point $z$ is \emph{escaping} if $\lim_{n\to \infty} |f^n(z)| = \infty$, and \emph{fast-escaping} if there is some $K$ for which $|f^{n-K}(z)| > |g^n(0)|$ for all $n\geq0$, where $g: z \mapsto e^z$. 
Regarding the latter conjecture, \emph{hairs} are connected components of the fast-escaping set. They were shown to admit $\cC^\infty$ parametrisations by Viana \cite{Viana:hairs}. 
The Julia set 
can be represented as the closure of the set of repelling periodic points. 
In this context, it coincides with the closure of the set of escaping points.  

From \cite{DobLSD, DobDense}, the third iterate of an oblique line is dense, under the exponential map. If a line segment intersects the fast-escaping set transversally, it is conceivable (and true, as we shall see) that the segment will have a dense orbit. 

Nice sets are important objects in complex dynamics (see for example, \cite{PrzRL11Rat, RLS14Stat, DobNice}) 
which generalise, in some sense,
Yoccoz puzzle pieces \cite{Mil00Local, Hub93Puz,Roe08Newt}.  
For every $z,\eps$, there is the smallest nice set $U(z,\eps)$ containing $B(z,\eps)$, so we have a canonical map $(z,\eps) \mapsto U(z,\eps)$. 
If there is some sort of backward contraction, the \emph{appending pullbacks method} \cite{RL07connect,PrzRL07Stat, DobNice} is powerful and guarantees that the set $U(z,\eps)$ will be comparable in size to $B(z,\eps)$. This method can also be used in conjunction with branch selection \cite{DobRat, DobNice}.
It is natural to wonder how regular is the boundary of a typical nice set $U(z,\eps)$.

\emph{Unbounded} nice sets with smooth boundaries exist. Indeed, for $z \mapsto e^z$, the real line is forward-invariant, so the upper and lower half-planes are nice sets, as are their pullbacks. 

For rational maps with real coefficients, the real line is forward-invariant. The upper and lower half-planes are nice sets, as are their pullbacks, but on the Riemann sphere (the natural domain for rational maps) all sets are bounded. Hence in this context, bounded nice sets with piecewise-smooth boundary exist. 
It would be interesting to know whether, modulo Moebius transformations, any \emph{other} rational maps with Julia set the whole sphere can have nice sets with piecewise-smooth boundary. 

\begin{example}
    Even in the exponential setting, part of the boundary of a bounded nice set may be smooth. Consider $f : z \mapsto e^z$. For $\eps >0$ small enough, \cite[Proposition~3]{DobNice} implies that the nice set $U(2,\eps)$ is bounded. 
    Let $V = U(2,\eps) \cap \mathbb{H}$ be the part of $U(2,\eps)$ lying in the upper half-plane. It is a nice set whose boundary contains a real interval. 
\end{example}

\section{Proofs}
Let $\lambda \in \ccc\setminus \{0\}$ and $f : z \mapsto \lambda e^z$  be a map from the exponential family. 
We denote by $\re(z)$ and $\im(z)$ the real and imaginary parts of a complex number $z$. 

\begin{lem}
    Given $\eps > 0$ and $z \in \ccc\setminus\{0\}$, there is $y \in \ccc\setminus\{0\}$ with $|\arg(y)-\pi/4| < \eps$ and $f^2(y) = z$. 
\end{lem}
\begin{proof}
For every $r$ large enough, the sector with argument $(\pi/4-\eps, \pi/4 + \eps)$ contains a vertical line segment of length $2\pi$ and real part $r$. 
A closed vertical line segment of length $2\pi$ and real part $r$ gets mapped by $f$ to a circle centred on the origin of radius $|\lambda|e^r$. 
The map $f$ is $2\pi i$-periodic, so any non-zero point has preimages arbitrarily far away from the origin. Combining these ideas, the lemma follows.
\end{proof}
\begin{lem}
    Given $\eps > 0$ and $y \in \ccc\setminus\{0\}$ with $|\arg(y)-\pi/4| < \frac12$,
     there are four closed arcs, $C_1,\ldots,C_4$ say, of circles centred at the origin 
     such that 
    \begin{itemize}
        \item
            $C_1 \cap C_2 = C_3 \cap C_4 = \emptyset$;
        \item
            the length of any arc is at most $\eps_1$;
        \item
            the union of the arcs $C_1, C_2$ and the $2\pi i$-translate of the arcs $C_3,C_4$ is a Jordan curve surrounding the point $y$.
    \end{itemize}
\end{lem}
\begin{proof}
    Evident.
\end{proof}
Combining the two preceding lemmas and knowing that $f$ is a local diffeomorphism, we obtain:
\begin{lem} \label{lemCircles}
    Given $\eps >0$ and $z \in \ccc\setminus\{0\}$, there are four closed arcs $C_1, \ldots, C_4$ of circles centred at the origin such that $f^2(C_1 \cup \cdots \cup C_4)$ is a Jordan curve surrounding $z$ with diameter bounded by $\eps$ and if $f^2(C_i) \cap f^2(C_j) \ne \emptyset$ and $i \ne j$, then $f^2(C_i)$ meets $f^2(C_j)$ transversally. 
\end{lem}
Hence if we can approximate circles, we can create piecewise-smooth Jordan curves surrounding any non-zero point. This will allow us to recursively find points with dense orbits. 

\medskip



Denote by $I$ the imaginary axis; 
the  set $f^{-1}(I)$ partitions the plane into horizontal strips of height $\pi$. Every second strip gets mapped by $f$ to the right half-plane.  Two such strips, $\hat S_0$ and $\hat S_1$ say,  are contained in $\{s : |\im(z)| \leq 5\pi/2\}$ (there is a third if $\lambda \in \R$, but its existence is not important to us).
We choose $K > 10$ large enough that $|\lambda|e^K > 200K$. If $\re(z) \geq K$, then $|f(z)| > 200 \re(z) > 2000$. 
Let
$$  
S_0 := \overline{\{z \in \hat S_0  \colon \re(z) \geq K\}}
\text{ and }
    S_1 := \overline{\{z \in \hat S_1  \colon \re(z) \geq K\}}.$$
    Denote the union $S_0 \cup S_1$ by $S$. 

    \begin{lem} \label{lemElt}
        For $n\geq 1$, if $z, f(z), \ldots, f^n(z) \in S$ then $\re(f^n(z)) > 100^n\re(z)$ and 
        $$|\arg(f^n(z))| < \frac{10}{100^n\re(z)},$$ 
        while 
        \begin{equation}\label{eqnarg}
        |\arg(Df^n(z))| \leq \sum_{j=1}^n|\arg(f^j(z))| < \frac{10}{\re(z)} \frac{1}{99} < 1/50.
    \end{equation}
    \end{lem}
\begin{proof} Elementary. \end{proof}

    
    Each connected component $S_0$ and $S_1$ of $S$ gets mapped by $f$ univalently onto the closure of the intersection of the right half-plane with the exterior of the disc of radius $|\lambda|e^K$.
Given a sequence $a = (a_n)_{n\geq0} \in  \{0,1\}^\N$, 
let $$T_a := \{z \in S \colon f^n(z) \in S_{a_n} \text{ for all } n\geq 0\}.$$
Then $T_a$ is a simple curve in $S$ which tends to $\infty$, admits a $\cC^\infty$ parametrisation \cite{Viana:hairs} 
and has an endpoint with real part $K$. 
By \eqref{eqnarg}, 
the tangent vectors to $T_a$ always have  argument with modulus bounded by $1/50$ and are asymptotically horizontal. 

We denote by $T$ the subset of the Julia set defined by
$$
T := \bigcup_{a \in \{0,1\}^\N} T_a = \{z \colon  f^n(z) \in S \text{ for all }n \geq 0\} = \bigcap_{n\geq0} f^{-n}(S).$$
The connected components of $T$ are the pairwise-disjoint curves $T_a, a \in \{0,1\}^\N$, of which there are uncountably many. 


\begin{lem} \label{lemNotcontained}
    Consider a continuous curve $\gamma : [0,1] \to S$ satisfying $\{0\}$ is a connected component of $\gamma^{-1}(T)$. There are arbitrarily small $\eps >0$ such that, for some $k \geq 1$,  $f^n(\gamma([0,\eps])) \in S$ for $n=1,\ldots, k$ and $f_k(\gamma(\eps)) \in \partial S$. 
\end{lem}
\begin{proof}
    For $a \in \{0,1\}^{\N}$, we can denote by
     $W^+_{0,a}$ the upper horizontal boundary of $S_{a_0}$ and $W^-_{0,a}$ the lower horizontal boundary. 
     Then let 
     $W_{k,a}^\pm$ be the set of points $z$ satisfying $f^n(z) \in S_{a_n}$ for $n = 0, 1,\ldots, k$ and $f^k(z) \in W_{0,\sigma^k(a)}^\pm$, where $\sigma  : \{0,1\}^\N \to \{0,1\}^\N$ denotes the left-shift. Each $W_{k,a}^\pm$ is a curve in $S$ joining the left boundary of $S$ to $\infty$. 
    The curves $W_{k,a}^+$ accumulate on $T_a$ from above and  $W_{k,a}^-$ accumulate on $T_a$ from below. The result follows, choosing $a$ so that $\gamma(0) \in T_a$.
\end{proof}

%
%
We say that a $\cC^1$ curve $\gamma : [0,1] \to \C$ intersects $T$ transversally at $\gamma(t)$ if $\gamma(t) \in T_a$ for some $a$ and $\gamma$ (restricted to any neighbourhood of $t$) intersects the curve $T_a$ transversally. 
Given a $\cC^1$ curve $\gamma : [0,1] \to \C$, we denote by $A_n(\gamma)$ the set of angles (in $[0,\pi)$) at which $f^n\circ \gamma$ intersects $T$ (transversally provided the angle is non-zero). As $T$ is forward invariant and smooth, $A_n(\gamma) \subset A_{n+1}(\gamma)$. 

\begin{prop} \label{propTransv}
    Let $\gamma : [0,1] \to S$ be a $\cC^1$ curve satisfying $\{0\}$ is a connected component of $\gamma^{-1} (T)$. 
    Then 
    $$
    \overline{ \bigcup_{n\geq0} A_n(\gamma) } = [0,\pi].$$
\end{prop}
\begin{proof}
    Let $C>1.$
    By replacing by an iterate and a subinterval if necessary, we may assume that ${\gamma([0,1])}\subset \{z : \re(z)>C\}$. 
    Obtain $k$ 
    and $\eps>0$ 
    from \rlem{Notcontained}, where $\eps$ is
     small enough that the argument of $\gamma'$ on $[0,\eps]$ varies by at most $1/C$.
     Now applying \rlem{Elt}, the argument of $(f^k\circ \gamma)'$ differs from $\gamma'(0)$  by at most $2/C$. 
     By choice of $k$ and $\eps$, $f^n(\gamma([0,\eps])) \subset S$ for $n \leq k$ and $f^k\circ \gamma(\eps) \in \partial S$. 
     By the Chain Rule,
     the argument of $(f^{k+1}\circ \gamma)'(t)$ is, modulo $2\pi$, within $3/C$ of 
     $$\gamma'(0) + \arg(f^{k+1}(\gamma(t))).$$
     The latter summand varies between $\arg(f^{k+1}(\gamma(0))) \approx 0$ and $\arg(f^{k+1}(\gamma(\eps))) = \pm \pi/2$. 
    Meanwhile, the curve cannot approach zero: 
    $$|f^{k+1}(\gamma(t))| > e^C$$ for $t \in [0,\eps]$. 
    Hence $f^{k+1}\circ \gamma_{|[0,\eps]}$ crosses each $2j \pi i $-translate of $T$ for $\alpha j = 1,2, \ldots, \lfloor e^C/2\pi \rfloor$ where $\alpha$ is either $1$ or $-1$. 
    The angles of intersection are $3/C$-dense in $[\arg(\gamma'(0)), \arg(\gamma'(0))\pm \pi/2]$. 
    
    We can now obtain a subcurve $\rho : [-1, 1] \to \gamma([0,\eps])$ so that $\{0\}$ is a connected component of $\rho^{-1}(T)$ and repeat the process but with $\alpha$ taking \emph{both} values $\pm 1$. As $C$ was arbitrary, the result follows. 
\end{proof}

We shall use little-$o$ notation, where any $o(1)$ term tends to $0$ as $n \to \infty$. 

\begin{lem} \label{lemCircleApprox}
    Let $r >0$ and $\rho_r : \xi \mapsto r e^{i\xi}$, defined on $[0,4\pi]$. 
    Let $\gamma : [0,1] \to S$ be a $\cC^1$ curve satisfying $\{0\}$ is a connected component of $\gamma^{-1} (T)$.

There are $n>0$ and  a closed subinterval of $(0,1]$ on which $f^n\circ \gamma$ is affinely conjugate to a map which is arbitrarily $\cC^1$-close to $\rho_r$.
\end{lem}

\begin{proof}
    Fix $\eps \in (0, 1/K)$ such that $\eps < r < \eps^{-1}$. 
    By \rprop{Transv}, we may assume that $\gamma$ intersects $T$ transversally at $\gamma(0)$ at an angle in $(2\eps - \eps^3, 2\eps + \eps^3)$.
    Let $J_n = [0,\delta_n)\subset [0,1]$ be such that the length of $f^n\circ \gamma (J_n)$ is $1/n$. 
        One can check, using \rlem{Elt}, that, for large $n$,
        $$\arg( (f^n \circ \gamma)'(J_n)) \subset (2\eps-2\eps^3, 2\eps + 2\eps^3) 
        $$
    and that 
    $$|\arg(f^{n+1}\circ \gamma (J_n))| < \frac1n < \eps^{-3}.$$
    Hence
    $$\arg( (f^{n+1} \circ \gamma)'(J_n)) \subset (2\eps-3\eps^3, 2\eps + 3\eps^3).
    $$
    The length of $f^{n+1} \circ \gamma(J_n)$ is at least $100^n /n > \eps^{-2}$ for large $n$. 
    Hence its imaginary part covers an interval of length $3\pi$, which guarantees the existence of a subarc $A \subset f^{n+1} \circ \gamma(J_n)$  containing a point $p$ with $\dist(p, \partial A) = 1/n$ and with the length of $A$ bounded by $3/n$, for which $f(p)$ lies on the positive imaginary axis. 
    The length bound controls the distortion of $f$ on $A$ and of $f^{n+1}$ on $J_n \cap f^{-n-1}(A)$.
    
Since the horizontal line passing through $p$ maps to the positive imaginary axis, 
$f(A)$ is almost vertical, with direction lying in $(\pi/2 + \eps, \pi/2 + 3\eps)$. Each component of $f(A\setminus \{p\})$ has length at least $100^n/n$ and so, for large $n$, the real part of $f(A)$ covers $(-n,-n)$. 
Hence some subarc $W'$ of $f(A)$ has a $2k\pi i$-translate which is $6\pi \eps$-close (in the Hausdorff metric) to the line segment joining $\log r$ and $4\pi i + \log r$. 
As $\eps>0$ was arbitrary, we can choose $W'$ so that  $f(W')$ is arbitrarily close to the circle of radius $r$.

Let $W$ be the subinterval of $J_n$ mapped by $f^{n+2}\circ\gamma$ into $A$ and thence by $f$ bijectively onto $W'$.
Noting that 
\[
    \sup_{w,z \in \gamma(W)}
    \frac{Df^{n+3}(w)}{Df^{n+3}(z)}  = 1 +o(1),
\]
$f^{n+4}\circ\gamma$ restricted to $W$, affinely reparametrised, is $\cC^1$-close to $\rho_r$.
\end{proof}

\begin{prop} \label{propcurvearcs}
    Let $\gamma : [0,1] \to S$ be a $\cC^1$ curve satisfying $\{0\}$ is a connected component of $\gamma^{-1} (T)$.
    Given $z \in \C \setminus \{0\}$ and $\eps >0$ there are four closed intervals $W_j \subset [0,1]$ and numbers $n_j \geq 1$, $j=1,\ldots, 4$, for which
    $$
    \bigcup_{j=1}^4 f^{n_j}\circ \gamma(W_j)$$
    is a (piecewise smooth) Jordan curve with diameter bounded by $\eps$ and containing $z$ in its interior region. 
    \end{prop}
    \begin{proof}
        This follows from \rlem{CircleApprox} and \rlem{Circles}.
    \end{proof}


\begin{lem} \label{lemMont}
Let $H$ be a Jordan curve.
Suppose that the interior region bounded by $H$ contains points in the Julia set.
There exists $N\geq 1$ such that $f^N(H)$ intersects each connected component of $T$ at a point with real part strictly greater than $K+1$.

Furthermore, given $\eps>0$,  $H$ contains an arc $A$ for which $f^N(A) \subset S \cap \{z \colon \re(z) > K+1\}$, for which $f^N(A)$ has diameter bounded by $\eps$ and  for which $f^N(A)$ intersects uncountably many connected components of $T$.
\end{lem}
\begin{proof}
    Take $\cK$ to be a closed, vertical, line segment joining $K+1 - {3\pi} i$ with $K+1+{3\pi} i$. By construction, it crosses both connected components of $S$ and, hence, all connected components of $T$. Each connected component of $T$ is a curve which crosses $\cK$ and joins $\cK$ with $\infty$. 
    
    As periodic points are dense in the Julia set, there is a repelling periodic point of period $k$, say, in the interior region bounded by $H$. The periodic point has an open neighbourhood $V$ for which $V \subset f^k(V)$ and $V \cap H = \emptyset$.
By Montel's theorem,  there exists $N\geq 1$ for which $\cK \subset f^N(V)$. 
The first statement follows. 

Without loss of generality, we suppose that $\eps \in (0,1)$. 
        As $f^N$ is locally diffeomorphic, we can partition $H$ into a finite number of arcs on each of which $f^N$ is homeomorphic onto its image, and whose image has length bounded by $\eps$. 
        Verifying that $\dist(T \cap \{z : \re(z) > K+1\}, \partial S) = 1$ is straightforward; then  apply the Pigeon-Hole Principle to conclude. 
\end{proof}
\begin{rem} This falls just short of implying that $H$ contains a Cantor set of points mapped by $f^N$ into $T$. 
\end{rem}

\begin{cor} \label{corDich}
    Either $H$ contains an uncountable set of points where it is not smooth or there is a $\cC^1$ curve $\rho : [0,1] \to H$ for which $f^N\circ\rho([0,1]) \subset S$ and $\{0\}$ is a connected component of $(f^N\circ\rho)^{-1}(T)$. 
\end{cor}
    \begin{proof}
        From \rlem{Mont}, we have an uncountable set 
        $$\cB := 
        \{ a \in \{0,1\}^\N \colon f^N(A\setminus \partial A) \cap T_a \ne \emptyset. \}.$$ 
        We can orient $A$ by embedding it in $\R$; then for $a \in \cB$, there is a rightmost point $s_a$ in $A$ for which $f^N(s_a) \in T_a$. 
        It follows that either $H$ is not smooth at $s_a$ for any $a \in \cB$, or there is a differentiable curve $\rho$ as required. 
    \end{proof}

        The following two corollaries are immediate consequences of \rcor{Dich} and \rprop{curvearcs}.

    \begin{cor}
        \rmthm{Dicho} holds.
    \end{cor}
\begin{cor} \label{corIndStep}
    Given $z \in \ccc\setminus\{0\}$, $\eps>0$ and a piecewise smooth Jordan curve $H$ 
    with a point of the Julia set in its interior region, there are four closed subarcs $W_1, \ldots, W_4$ of $H$ and numbers $n_1,\ldots, n_4 \geq 1$ such that the union of $f^{n_j} \circ h_{|W_j}$ is a piecewise smooth Jordan curve with diameter bounded by $\eps$ and containing $z$ in its interior region. 
\end{cor}
    When the Julia set is the whole complex plane, and noting that the intersection of a decreasing sequence of closed sets is non-empty, we can repeatedly apply \rcor{IndStep} for a dense sequence $(z_n)_n$ of (non-zero) points and a sequence $(\eps_n)_n$ tending to $0^+$ to obtain:
    \begin{cor}
        \rmthm{Point} holds. 
    \end{cor}

\bibliography{references}
\bibliographystyle{plain}
\end{document}